\documentclass[11pt,twoside,reqno,centertags,draft]{amsart}
\usepackage{color}

  \usepackage{amsmath,amsthm,amsfonts,amssymb}

\begin{document}

\title[The Nehari manifold approach to singular problems]{The Nehari manifold approach  for  singular  equations involving the $p(x)$-Laplace operator}
\thanks{{\it Math. Subj. Classif.} (2010):  35J20, 35J60, 35J70, 47J10}

\maketitle

\vspace{ -1\baselineskip}

{\small
\begingroup\small
\begin{center}

{\sc Du\v{s}an D. Repov\v{s}}\\

 Faculty of Education and Faculty of Mathematics and Physics, University of Ljubljana \&
Institute of Mathematics, Physics and Mechanics,  1000 Ljubljana, Slovenia \\

Email: {\sl dusan.repovs@guest.arnes.si}
\\[10pt]

{\sc Kamel Saoudi}\\
 
  College of Sciences  at Dammam, University of Dammam\\
 31441  Dammam, Kingdom of Saudi Arabia \\
Email: {\sl  kmsaoudi@iau.edu.sa}  
\\[10pt] 

\end{center}
\endgroup

\numberwithin{equation}{section}
\allowdisplaybreaks

\newtheorem{theorem}{Theorem}[section]
\newtheorem{proposition}{Proposition}[section]
\newtheorem{gindextheorem}{General Index Theorem}[section]
\newtheorem{indextheorem} {Index Theorem}[section]
\newtheorem{standardbasis}{Standard Basis}[section]
\newtheorem{generators} {Generators}[section]
\newtheorem{lemma} {Lemma}[section]
\newtheorem{corollary}{Corollary}[section]
\newtheorem{example}{Example}[section]
\newtheorem{examples}{Examples}[section]
\newtheorem{exercise}{Exercise}[section]
\newtheorem{remark}{Remark}[section]
\newtheorem{remarks}{Remarks}[section]
\newtheorem{definition} {{Definition}}[section]
\newtheorem{definitions}{Definitions}[section]
\newtheorem{notation}{Notation}[section]
\newtheorem{notations}{Notations}[section]
\newtheorem{defnot}{Definitions and Notations}[section]

\def\H{{\mathbb H}}
\def\N{{\mathbb N}}
\def\R{{\mathbb R}}

\newcommand{\be} {\begin{equation}}
\newcommand{\ee} {\end{equation}}
\newcommand{\bea} {\begin{eqnarray}}
\newcommand{\eea} {\end{eqnarray}}
\newcommand{\Bea} {\begin{eqnarray*}}
\newcommand{\Eea} {\end{eqnarray*}}
\newcommand{\p} {\partial}
\newcommand{\ov} {\over}
\newcommand{\al} {\alpha}
\newcommand{\ba} {\beta}
\newcommand{\de} {\delta}
\newcommand{\ga} {\gamma}
\newcommand{\Ga} {\Gamma}
\newcommand{\Om} {\Omega}
\newcommand{\om} {\omega}
\newcommand{\De} {\Delta}
\newcommand{\la} {\lambda}
\newcommand{\si} {\sigma}
\newcommand{\Si} {\Sigma}
\newcommand{\La} {\Lambda}
\newcommand{\no} {\nonumber}
\newcommand{\noi} {\noindent}
\newcommand{\lab} {\label}
\newcommand{\na} {\nabla}
\newcommand{\vp} {\varphi}
\newcommand{\var} {\varepsilon}
\newcommand{\RR}{{\mathbb R}}
\newcommand{\CC}{{\mathbb C}}
\newcommand{\NN}{{\mathbb N}}
\newcommand{\ZZ}{{\mathbb Z}}
\renewcommand{\SS}{{\mathbb S}}

\newcommand{\esssup}{\mathop{\rm {ess\,sup}}\limits}   
\newcommand{\essinf}{\mathop{\rm {ess\,inf}}\limits}   
\newcommand{\weaklim}{\mathop{\rm {weak-lim}}\limits}
\newcommand{\wstarlim}{\mathop{\rm {weak^\ast-lim}}\limits}

\newcommand{\RE}{\Re {\mathfrak e}}   
\newcommand{\IM}{\Im {\mathfrak m}}   
\renewcommand{\colon}{:\,}
\newcommand{\eps}{\varepsilon}
\newcommand{\half}{\textstyle\frac12}
\newcommand{\Takac}{Tak\'a\v{c}}
\newcommand{\eqdef}{\stackrel{{\rm {def}}}{=}}   
\newcommand{\wstarconverge}{\stackrel{*}{\rightharpoonup}}

\newcommand{\Div}{\nabla\cdot}
\newcommand{\Curl}{\nabla\times}
\newcommand{\Meas}{\mathop{\mathrm{meas}}}
\newcommand{\Int}{\mathop{\mathrm{Int}}}
\newcommand{\Clos}{\mathop{\mathrm{Clos}}}
\newcommand{\Lin}{\mathop{\mathrm{lin}}}
\newcommand{\Dist}{\mathop{\mathrm{dist}}}

\newcommand{\Square}{$\sqcap$\hskip -1.5ex $\sqcup$}
\newcommand{\Blacksquare}{\vrule height 1.7ex width 1.7ex depth 0.2ex }

\smallskip
\begin{quote}
\footnotesize
\begin{center}
 {\bf Abstract}
\end{center}
We study  the following singular  problem involving
the p$(x)$-Laplace operator
$\Delta_{p(x)}u= div(|\nabla u|^{p(x)-2}\nabla u)$,
where  $p(x)$ is a  nonconstant
 continuous function,
\begin{equation}
\nonumber
{{(\rm P_\la)}}
\left\{
\begin{aligned}
  - \Delta_{p(x)} u 
& = a(x)|u|^{q(x)-2}u(x)+ \frac{\lambda b(x)}{u^{\delta(x)}}
  \quad\mbox{ in }\,\Omega,\\
  u&>0 \quad\mbox{ in }\,\Omega,
\\
  u
& =0  \quad\mbox{ on }\,\partial\Omega.
\end{aligned}
\right.
\end{equation}
Here,  
 $\Omega$ is a bounded domain in $\R^{N\geq2}$ with $C^2$-boundary, 
$\lambda$ is a positive parameter, $a(x), b(x) \in  C(\overline{\Omega})$
are positive weight functions with compact support in $\Omega­$, 
and   
  $\delta(x),$ $p(x),$ $q(x) \in C(\overline{\Omega})$ satisfy
certain  hypotheses ($A_{0}$) and ($A_{1}$). We
apply 
 the Nehari manifold approach and some
 new
  techniques
  to establish the multiplicity of positive solutions
for problem ${{(\rm P_\la)}}$. 

\end{quote}

\section{Introduction} \label{intro}
The aim of this paper is to study the
 following inhomogeneous   equation
\begin{equation}
\nonumber
{{(\rm P_\la)}}
\left\{
\begin{aligned}
  - \Delta_{p(x)} u 
& =  a(x)|u|^{q(x)-2}u(x)+ \frac{\lambda b(x)}{u^{\delta(x)}}
  \quad\mbox{ in }\,\Omega,\\
  u&>0 \quad\mbox{ in }\,\Omega,
\\ u
& =0  \quad\mbox{ on }\,\partial\Omega.
\end{aligned}
\right.
\end{equation}
Here,  operator  
$\Delta_{p(x)}u := div(|\nabla u|^{p(x)-2}\nabla u)$  
is  the
$p(x)$-Laplacian,
 $p(x)$ is a  nonconstant
 continuous function, 
$\Omega$ is a bounded domain in $\R^N$ $(N\geq 2)$ with $C^2$-boundary, 
$\lambda$ is a positive parameter, 
$a(x), b(x) \in  C(\overline{\Omega})$ are positive weight functions with compact support in $\Omega­,$ 
  and 
 $\delta(x),p(x),q(x) \in C(\overline{\Omega})$ satisfy the following  conditions
\begin{itemize}
\item[{ ($A_{0}$)}] 
  $0<1-\delta(x)<p(x)<q(x)<p^*(x);$\\
\item[{ ($A_{1}$)}] 
$0<1-\delta^- \leq 1-\delta^+<p^-\leq p^+<q^-\leq q^+.$ 
\end{itemize}
 Here,
 $
 p^\ast (x):= Np(x)/(N-p(x)), 
 \delta^{+}:= \esssup \delta(x),
  \delta^{-}:= \essinf \delta(x),
 $
 and analogous definitions hold for 
 $p^-, p^+, q^-,$ and $  q^+.$

Partial differential equations  with variable exponents  are a  very interesting and active topics. The motivation for this type of problems was stimulated by their various applications in physics - for more details see 
{\sc Acerbi-Mingione}~\cite{AcMi}, 
{\sc Diening}~\cite{Di}, and in particular the book
{\sc R\u{a}dulescu-Repov\v{s}}~\cite{RR}) 
and the references therein. 

Before stating our main result, we review the key  literature concerning singular partial differential equations  with variable exponents.
{\sc Zhang}~\cite{Zh1}
 proved the existence of solutions  for the purely singular problem. 
  Using variational methods,
 {\sc Saoudi}~\cite{Sa5}
  proved the existence for a superlinear singular equation with variable exponent.
   {\sc  Fan}~\cite{F2}
    investigated the multiplicity of solutions using topological methods. In
     {\sc Saoudi-Ghanmi}~\cite{SaGh} 
     and
       {\sc Saoudi et al.}~\cite{SaKrAl}
         variational methods 
  were used to establish the  multiplicity of solutions for a singular  problems with Dirichlet and  Neumann conditions, respectively
  (see also  {\sc Saoudi-Ghanmi}~\cite{SaGh1}).

The case when $p$ is constant in problem  ${\rm (P_\la)},$ 
has
  recieved more attention and has been approached by
  differents techniques. For a more general
presentation
we refer to   
 {\sc Coclite-Palmieri}~\cite{CoPa},
 {\sc  Crandall et al.}~\cite{CrRaTa}, 
 {\sc Ghergu-R\u{a}dulescu}~\cite{GhRa}, 
 {\sc  Gia\-co\-mo\-ni-Saoudi}~\cite{GiSa}, 
 {\sc Giacomoni et al.}~\cite{GiScTa}, 
 {\sc Saoudi }~\cite{Sa3},
and
 {\sc  Saoudi-Kratou}~\cite{SaKr} and the references therein. 

Some interesting papers on the applications of the Nehari manifold method in variable exponent problem have recently been published, see e.g.
 {\sc Mashiyev et al.}~\cite{MaOgYuAv},   
  {\sc Saiedinezhad-Ghaemi}~\cite{SaiGh}, and {\sc Saoudi}~\cite{Sa4}.
In the present paper, we  generalize the results  of 
{\sc Giacomoni et al.}~\cite{GiScTa}
 and {\sc Saoudi}~\cite{Sa3}
  to
  the problem with variable exponent,  by using topological methods. 
Here is the main result of this paper.

\begin{theorem}\label{th1} Suppose that condition {($A_{0}$)}
and {($A_{1}$)} are fulfilled. Then there exists $\lambda_0>0$ such that 
 for every $\lambda\in(0,\lambda_0),$
problem $({\rm P}_\la)$
 has at least two positive solutions.
\end{theorem}

This paper
is  organized as follows. 
In Section \ref{prelim},
 we  briefly review the properties of generalized Lebesgue-Sobolev spaces. 
In Section \ref{fibering}, 
 we  prove the necessary lemmas. 
In Section \ref{existence1},
 we prove the existence of a minimum for the functional energy $E_\la$ in
$\mathcal{N}_{\lambda}^{+}$.
In Section \ref{existence2}, 
we  prove the existence of a minimum for the functional energy $E_\la$ in
$\mathcal{N}_{\lambda}^{-}$. 
Finally, In Section \ref{proof},
we present the proof of our main result.

\section{Generalized Lebesgue-Sobolev Spaces}\label{prelim}
In this section, we recall   definitions of functional spaces with variable exponents  and properties of the  $p(x)$-Laplacian operator  which will be used later (for more on this topics see 
{\sc R\u{a}dulescu-Repov\v{s}}~\cite{RR}, and for other additional information see
{\sc Papageorgiou et al.}~\cite{PRR}).
Let
\[
L^{p(\cdot )}( \Omega ) =\big\{  u\in S( \Omega
) : \int_{\Omega }| u(x)| ^{p(x)}dx<\infty \big\} ,
\]
with the norm
\[
| u| _{p(\cdot )}=| u| _{L^{p(\cdot
)}(\Omega )}=\inf \big\{ \lambda >0 : \int_{\Omega }\left|
\frac{u(x)}{\lambda }\right| ^{p(x)}dx\leq 1 \big\} .
\]
Then $( L^{p(\cdot )}( \Omega ) ,|\cdot | _{p(\cdot )}) $  
 is a reflexive, uniform convex Banach, separable space - for details see
{\sc Fan-Zhao}~\cite[Theorems 1.6, 1.10, 1.14]{f8}
and
{\sc R\u{a}dulescu-Repov\v{s}}~\cite{RR}.

The variable exponent Sobolev space
\[
W^{1,p(\cdot )}( \Omega ) =\big\{ u\in L^{p(\cdot )}(
\Omega ) : | \nabla u| \in L^{p(\cdot)}(\Omega )\big\} ,
\]
can be equipped with the norm
\[
\| u\| =| u|_{p(\cdot )}+| \nabla u| _{p(\cdot )},\quad \hbox{for all} \ 
u\in W^{1,p(\cdot )}( \Omega ) .
\]
Note that  $W_0^{1,p(\cdot )}( \Omega ) $ is the closure of
$C_0^{\infty}( \Omega ) $ in $W^{1,p(\cdot )}( \Omega ) $.

We denote by $L^{q(x)}(\Omega)$ the
 conjugate space of $L^{p(x)}(\Omega),$ where $\frac{1}{q(x)}+\frac{1}{p(x)}=1.$ For
$u\in L^{p(x)}(\Omega)$ and $v\in L^{q(x)}(\Omega)$, the H\"older type inequality 
\begin{equation}\label{e2}
\Big{|}\int_\Omega u(x)v(x) dx\Big{|}\leq
\Big{(}\frac{1}{p^-}+\frac{1}{q^-}\Big{)}|u|_{{p(x)}}|v|_{{q(x)}},
\end{equation}
holds.
Recall the following result.
\begin{lemma}
Consider the mapping $\rho_{p(x)}:\,L^{p(x)}(\Omega)\rightarrow \mathbb{R}$
defined by
$$ \rho_{p(x)}(u)=\int_\Omega |u|^{p(x)} dx,$$
where $(u_n), u \in L^{p(x)}(\Omega),$ and $p^+<\infty.$ Then the following relations hold
\begin{equation}\label{e3}
\|u\|_{L^{p(x)}}>1 \Rightarrow \|u\|_{L^{p(x)}}^{p^-}\leq
\rho_{p(x)}(u) \leq \|u\|_{L^{p(x)}}^{p^+},
\end{equation}
\begin{equation}\label{e4}
\|u\|_{L^{p(x)}}< 1 \Rightarrow \|u\|_{L^{p(x)}}^{p^+}\leq
\rho_{p(x)}(u) \leq \|u\|_{L^{p(x)}}^{p^-},
 \end{equation}
\begin{equation}\label{e5}
\|u_n-u\|_{L^{p(x)}}\rightarrow0\;\;\mbox{if and only
if}\;\,\rho_{p(x)}(u_n-u)\rightarrow0.
\end{equation}
\end{lemma}
We state the Sobolev embedding theorem.
\begin{theorem}[ See 
{\sc Fan et al.}~\cite{Fan1} and
{\sc Kov\u{a}\v{c}ik-R\u{a}kosnik}~\cite{Kovacik}]\label{Imbedding}Let $p\in C(\bar{\Omega})$ with $p(x)> 1$ for each $x\in \bar{\Omega}$ where  $\Omega\subset  {\mathbb R}^N$ is an open bounded 
domain with Lipschitz boundary and  suppose   that
$p(x) \leq r(x) \leq p^{*}(x)$  and $r \in C(\bar{\Omega}),$ for all $x\in\overline{\Omega}.$ Then the  embedding $W^{1,p(x)}(\Omega)
\hookrightarrow L^{r(x)}(\Omega)$ is continuous. 
 Also,  if $r(x) < p^{*}(x)$ almost everywhere in $\overline{\Omega}$, 
 then
this embedding is compact.

\end{theorem}

Let $\rho (x,s)$ be a Carath\'eodory function  satisfying the following condition 
\begin{equation} \label{e2b}
| \rho (x,s)| \leq A\quad \text{for a.e. $x\in \Omega$
and all }s\in [ -s_0,s_0], 
\end{equation}
where $s_0>0$ and $A$ is a constant.
Recall  the following comparison principle.

\begin{lemma}[{\sc Zhang} {\cite[Lemma 2.3]{z1}}]  \label{lem2.4}
Let $\rho (x,t)$ be a function satisfying \eqref{e2b} and increasing in $t$.
Let  $u,v\in W^{1,p(\cdot )}(\Omega )$ satisfy
\[
-\Delta _{p(x)}u+\rho (x,u)\leq -\Delta _{p(x)}v+\rho (x,v),\quad
\hbox{for all} \  \
 x\in \Omega.
\]
and assume that $u\leq v$ on $\partial \Omega $. Then $u\leq v$ in $\Omega $.
\end{lemma}
Next, we recall the
following  strong maximum principle. 
\begin{theorem}[ {\sc  Saoudi-Ghanmi} {\cite[Theorem 3.2]{SaGh}}] Suppose that  for some $0 < \alpha < 1$,
$u,v\in C^{1,\alpha}(\overline{\Omega})$
we have
$0\lneqq u$, $0\lneqq v$, and
\begin{align}
\label{u}
  -\Delta_{p(x)} u-\frac{\lambda}{u^{\delta(x)}} = h(x)\geq g(x) = -\Delta_{p(x)} v-\frac{\lambda}{v^{\delta(x)}},
\end{align}
with $u = v = 0$ on $\partial\Omega$,
where $g,h\in L^\infty(\Omega)$ are such that\/
$0\leq g < h$ pointwise everywhere in $\Omega$. Assume that
\begin{equation}
 \frac{\partial u}{\partial {\bf n}}>0\quad
   \frac{\partial v}{\partial {\bf n}} > 0
    \;\mbox{ on }\; \partial\Omega,
\end{equation}
where ${\bf n}$ is the inward unit normal on $\partial\Omega.$ Then
 the following strong comparison principle holds:
\begin{equation}
\label{iscp}
  u > v \;\mbox{ in }\; \Omega, \quad\mbox{ and there is a positive  }\;
\epsilon \;\;\textrm{such that}\;\;    \frac{\partial (u-v)}{\partial \bf{n}}
  \geq\epsilon \
    \;\mbox{ on }\; \partial\Omega .
\end{equation}
\end{theorem}
We shall now prove the following result.
\begin{theorem}\label{Inj2}
Suppose that the domain    $\Omega$­ has  the cone property
and  consider $p\in C(\overline{\Omega}).$ Assume
that $b\in L^{\alpha(x)},\,b(x)>0$ for 
$x\in \Omega,\, \alpha\in C(\overline{\Omega})$ and $\alpha^->1,\;\alpha_0^-\leq \alpha_0(x)\leq \alpha_0^+\;
(\frac{1}{\alpha(x)}+\frac{1}{\alpha_0(x)}=1),$  $\delta\in C(\overline{\Omega}),$ and
\begin{equation}\label{inj1}
0<1-\delta(x)<\frac{\alpha(x)-1}{\alpha(x)}p^*(x),\;\ \hbox{for all} \  x\in \overline{\Omega}.
\end{equation} 
Then the embedding $W^{1,p(x)}(\Omega)\hookrightarrow L^{1-\delta(x)}_{b(x)}(\Omega)$ is compact. Moreover, there
is a constant $c_2> 0$ such that the following inequality holds
\begin{equation}\label{31eq}
\int_\Omega b(x)|u|^{1-\delta(x)}dx\leq  c_2(||u||^{1-\delta^-}+||u||^{1-\delta^+}).
\end{equation}
\end{theorem}
\begin{proof}
The proof of the first assertion
  is adopted from {\sc Fan} \cite{F2}. Let $u\in W^{1,p(x)}(\Omega)$ and let
$$r(x)=\frac{\alpha(x)}{\alpha(x)-1}(1-\delta(x))=\alpha_0(x)(1-\delta(x)).$$ Hence, \eqref{inj1} implies
 that $r(x)<p^*(x).$ Therefore, using Theorem \eqref{Imbedding}, we obtain
 $W^{1,p(x)}(\Omega)\hookrightarrow L^{r(x)}(\Omega).$ So, for $u\in W^{1,p(x)}(\Omega),$
 we get $|u|^{1-\delta(x)}\in L^{\alpha_0(x)}(\Omega).$ By \eqref{e2},
$$ \int_\Omega b(x)|u|^{1-\delta(x)}dx\leq  c_1|b|_{\alpha(x)}\left||u|^{1-\delta(x)}\right|<\infty.$$
This means that $W^{1,p(x)}(\Omega)\subset L^{1-\delta(x)}(\Omega).$ 

On the other hand, if $ u_n\rightharpoonup 0\;\textrm{weakly in}\; W^{1,p(x)}(\Omega),$ then we have that 
$ u_n\to 0\;\textrm{strongly in}\; L^{r(x)}(\Omega).$
Therefore,
$$ \int_\Omega b(x)|u_n|^{1-\delta(x)}dx\leq  c_1|b|_{\alpha(x)}\left||u_n|^{1-\delta(x)}\right|\to 0,$$
hence $|u_n|_{1-\delta(x),b(x)}\to 0$ and  we can conclude that
$$W^{1,p(x)}(\Omega)\hookrightarrow L^{1-\delta(x)}_{b(x)}(\Omega).$$

Next, we shall
prove  inequality \eqref{31eq}. First, we have from above
$$ \int_\Omega b(x)|u|^{1-\delta(x)}dx\leq  c_1|b|_{\alpha(x)}\left||u|^{1-\delta(x)}\right|<\infty.$$
Since $1-\delta^-\leq 1-\delta(x)\leq 1-\delta^+$ and 
$|u|^{1-\delta(x)}\leq |u|^{1-\delta^-}+|u|^{1-\delta^+},$ we obtain
$$ \int_\Omega b(x)|u|^{1-\delta(x)}dx\leq \int_\Omega b(x)|u|^{1-\delta^-}dx+\int_\Omega b(x)|u|^{1-\delta^+}dx.$$

On the other hand, using \eqref{e2}, \eqref{e3}, \eqref{e4}, \eqref{e5}, and condition
$p(x)<(1-\delta^-)\alpha_0(x)\leq (1-\delta^+)\alpha_0(x)<p^*(x),$ we get 
\begin{equation}\label{inj2}\int_\Omega b(x)|u|^{1-\delta^-}dx\leq c_2|b|_{\alpha(x)}\left||u|^{1-\delta(x)}\right|_{\alpha_0(x)}
=c_2|b|_{\alpha(x)}|u|^{1-\delta^-}_{(1-\delta^-)\alpha_0(x)}\leq c_3||u||^{1-\delta^-}.
\end{equation}

In the same way, one gets
\begin{equation}\label{inj3}\int_\Omega b(x)|u|^{1-\delta^+}dx\leq c_4||u||^{1-\delta^+}.
\end{equation}

Hence, using \eqref{inj2} and \eqref{inj3}, we have 
\begin{equation*}\int_\Omega b(x)|u|^{1-\delta(x)}dx\leq c_5(||u||^{1-\delta^-}+
||u||^{1-\delta^+}).
\end{equation*}
which completes the proof  of Theorem \ref{Inj2}.
\end{proof}
\begin{theorem}\label{Inj3}
Let $p\in C(\overline{\Omega})$, suppose the boundary of  domain $\Omega$­ has the cone property
and let $u \in W^{1,p(x)}(\Omega­).$ 
Then
there exist nonnegative  constants $c_6,\, c_7,\,
c_8,\, c_9 > 0$ such that the following inequalities hold
\begin{equation}
\nonumber
\int_\Omega a(x)|u|^{q(x)}dx\leq
\left\{
\begin{aligned}
  c_6|| u ||^{q^+}
  \quad\mbox{ if }\,||u||>1,\\
  c_7|| u ||^{q^-}
  \quad\mbox{ if }\,||u||<1.
\end{aligned}
\right.
\end{equation}
\begin{equation}
\nonumber
\int_\Omega b(x)|u|^{1-\delta(x)}dx\leq
\left\{
\begin{aligned}
  c_8|| u ||^{1-\delta^-}
  \quad\mbox{ if }\,||u||>1,\\
  c_9|| u ||^{1-\delta^+}
  \quad\mbox{ if }\,||u||<1.
\end{aligned}
\right.
\end{equation}
\end{theorem}
\begin{proof}
Theorem \ref{Inj3}  follows immediately by 
  {\sc Mashiyev et al.}~\cite[Theorem 2.3]{MaOgYuAv}  and Theorem \ref{Inj2}.
\end{proof}

\section{ Some necessary lemmas}\label{fibering}
Let us  define the  functional 
$E_\la: W^{1,p(x)}_0(\Omega)\to\R$  by
\begin{eqnarray}
&&
E_\la(u)\eqdef\int_\Omega\frac{|\nabla u|^{p(x)}}{p(x)} \,dx
-\int_{\Omega}\frac{a(x)|u|^{q(x)}}{q(x)}\,d x-\lambda\int_{\Omega}\frac{b(x)(u^+)^{1-\delta(x)}}{1-\delta(x)}\,d x. \label{energy}
\end{eqnarray}
 
\begin{definition} We say that $u\in W^{1,p(x)}_0(\Omega)$ is  a generalized solution of the equation
\begin{equation}\label{def1}
- \Delta_{p(x)} u 
 =  a(x)|u|^{q(x)-2}u(x)+ \frac{\lambda b(x)}{u^{\delta(x)}}
\end{equation}
if
 for all $\varphi \in C^\infty_0(\Omega)$ and  $\underset{K}{\essinf} \, u > 0$ for every compact set $K \subset\Omega,$ 
\begin{eqnarray}
\int_\Omega|\nabla u|^{p(x)-2}\nabla u \nabla\varphi \,dx=\int_{\Omega} a(x)|u|^{q(x)-1} \varphi\,d x+\lambda\int_{\Omega}b(x)u^{-\delta(x)} \varphi\,d x
\label{def2} 
\end{eqnarray}
for all $\varphi \in C^\infty_0(\Omega).$ 
\end{definition}
Obviously, every weak solution of problem $(P_\la)$ is also a generalized solution
of  equation \eqref{def1}.

In many problems, such as $(P_\la)$, $E_\la$ is not bounded below on $W^{1,p(x)}_0(\Omega),$ but it is
bounded below on the corresponding Nehari manifold which is defined by
$$\mathcal{N}_{\lambda}:=\{u\in W^{1,p(x)}_0(\Omega)\setminus\{0\}:\langle E'_{\lambda}(u),u\rangle=0\}.$$
Then
 $u\in \mathcal{N}_{\lambda}$ if and only if
\begin{equation}\label{1}
\int_\Omega\frac{|\nabla u|^{p(x)}}{p(x)} \,dx
-\int_{\Omega}\frac{a(x)|u|^{q(x)}}{q(x)}\,d x-\la\int_{\Omega}\frac{b(x)|u|^{1-\delta(x)}}{1-\delta(x)}\,d x=0.
\end{equation}
We note that $\mathcal{N}_{\lambda}$ contains every  solution of problem  $(P_{\lambda})$.\\
 
It is well-known that the Nehari manifold is closely related to the behavior of the functions $\Phi_{u} : [0,\infty)\rightarrow \mathbb{R}$ defined as $\Phi_{u}(t)=E_{\lambda}(tu).$
Such maps are called fiber maps and were introduced by {\sc Drabek-Pohozaev} \cite{DrPo}.
For $u\in W^{1,p(x)}_0(\Omega)\setminus\{0\}$, we define
\begin{eqnarray*}
  \Phi_{u}(t) &=&  \int_\Omega\frac{t^{p(x)}|\nabla u|^{p(x)}}{p(x)}\,{\rm d }x
 -\int_\Omega \frac{a(x)t^{q(x)}}{q(x)} |u|^{q(x)} \,dx-\la\int_{\Omega}\frac{b(x)t^{1-\delta(x)}|u|^{1-\delta(x)}}{1-\delta(x)}\,dx, \\
  \Phi'_{u}(t) &=& \int_\Omega t^{p(x)-1}|\nabla u|^{p(x)}\,{\rm d }x
 -\int_\Omega a(x) t^{q(x)-1} |u|^{q(x)} \,dx-\la\int_{\Omega}b(x) t^{-\delta(x)}|u|^{1-\delta(x)}\,dx, \\
  \Phi''_{u}(t) &=& \int_\Omega(p(x)-1) t^{p(x)-2}|\nabla u|^{p(x)}\,{\rm d }x
 -\int_\Omega a(x)(q(x)-1) t^{q(x)-2} |u|^{q(x)} \,dx\\&+&\la\int_{\Omega}b(x)\delta(x)t^{-\delta(x)-1}|u|^{1-\delta(x)}\,dx.
\end{eqnarray*}

It is easy to see that $tu \in \mathcal{N}_{\lambda}$ if and only if $\Phi'_{u}(t) = 0$ and in particular, $u \in \mathcal{N}_{\lambda}$ if and only if $\Phi'_{u}(1) = 0$. Thus it is natural to split $\mathcal{N}_{\lambda}$  into three parts corresponding to local
minima, local maxima and points of inflection defined
as follows:
$$\mathcal{N}_{\lambda}^{+}:=\{u\in \mathcal{N}_{\lambda}:\Phi''_{u}(1)>0\}=\left\{tu\in W^{1,p(x)}_0(\Omega)\setminus\{0\}: \Phi'_{u}(t)=0, \Phi''_{u}(t)>0\right\},$$
$$\mathcal{N}_{\lambda}^{-}:=\{u\in \mathcal{N}_{\lambda}:\Phi''_{u}(1)<0\}=\left\{tu\in W^{1,p(x)}_0(\Omega)\setminus\{0\}: \Phi'_{u}(t)=0, \Phi''_{u}(t)<0\right\},$$
$$\mathcal{N}_{\lambda}^{0}:=\{u\in \mathcal{N}_{\lambda}:\Phi''_{u}(1)=0\}=\left\{tu\in W^{1,p(x)}_0(\Omega)\setminus\{0\}: \Phi'_{u}(t)=0, \Phi''_{u}(t)=0\right\}.$$

Our first result is the following.
\begin{lemma}\label{bound} $E_{\lambda}$ is coercive and bounded below on  $\mathcal{N}_{\lambda}$.
\end{lemma}
\begin{proof}
Let $u\in  \mathcal{N}_{\lambda}$ and $||u||>1.$ Then, using \eqref{e3}-\eqref{e5} and  the embeddings from Theorem \ref{Imbedding}, 
we estimate $E_\la(u)$  as follows:

\begin{eqnarray*}
 E_\lambda(u)&& = \int_\Omega\frac{|\nabla u|^{p(x)}}{p(x)} \,dx
-\int_{\Omega}\frac{a(x)|u|^{q(x)}}{q(x)}\,d x-\la\int_{\Omega}\frac{b(x)|u|^{1-\delta(x)}}{1-\delta(x)}\,d x\nonumber\\
 &&\geq\left(\frac{1}{p^+}-\frac{1}{q^-}\right)\int_\Omega|\nabla u|^{p(x)}\,dx
 - \la \left(\frac{1}{1-\delta^+}-\frac{1}{q^-}\right)  \int_\Omega b(x)|u|^{1-\delta(x)}\,dx \nonumber\\
 &&\geq\left(\frac{1}{p^+}-\frac{1}{q^-}\right)\parallel  u \parallel^{p^-}
  -  \la c_8\left(\frac{1}{1-\delta^+}-\frac{1}{q^-}\right)||u||^{1-\delta^+}. 
\end{eqnarray*}

Note that since $0<\delta^+<1$ and  $1-\delta^+<p^-,$ 
it follows that $E_\lambda(u) \to \infty $
 as $||u||\to\infty.$ Therefore $E_\la$ is coercive and
bounded below. 
\end{proof}
\begin{lemma}\label{lem1}
 Let $u$ be a local minimizer for $E_{\lambda}$  on subsets $\mathcal{N}_{\lambda}^{+}$ or $\mathcal{N}_{\lambda}^{-}$ of $\mathcal{N}_{\lambda}$ such that $u\not\in \mathcal{N}_{\lambda}^{0}.$ Then $u$ is a critical point of $E_{\lambda}.$
\end{lemma}
\begin{proof}Recall $u$ is a local  minimizer for $E_{\lambda}$ under the constraint 
\begin{equation}\label{eq:Nota}
I_{\lambda}(u) := \langle E'_{\lambda}(u),u\rangle= 0.
\end{equation}
Hence, using the theory of Lagrange multipliers, we obtain the existence of   $\mu\in \mathbb{R}$ such that $$E'_{\lambda}(u)=\mu I'_{\lambda}(u).$$
Therefore,
$$\langle E'_{\lambda}(u),u\rangle=\mu \langle I'_{\lambda}(u),u\rangle =\mu \Phi''_{u}(1)=0.$$
So,
$u\not\in \mathcal{N}_{\lambda}^{0},$ hence $\Phi''_{u}(1)\neq 0$. Consenquently, $\mu = 0.$  The proof of  Lemma \ref{lem1} is
thus complete.
\end{proof}
\begin{lemma}\label{lem2.2} There exists $\la_0$ such that for every
$0<\la<\la_0,$ we have  $\mathcal{N}_{\lambda}^{\pm}\neq\emptyset$
and $\mathcal{N}_\la^{0}=\{0\}$.
\end{lemma}
\begin{proof}
First, by Lemma \ref{lem1}, we deduce that  $\mathcal{N}_\lambda^{\pm}$ are nonempty for $\lambda\in(0,\la_0).$   Now, suppose that there exists $u\in\mathcal{N}_{\lambda}^{0}$ such that $||u||>1.$  Using the definition of $\mathcal{N}_\la^{0},$ we obtain
$$
\int_\Omega|\nabla u|^{p(x)} \,dx
-\int_{\Omega}a(x)|u|^{q(x)}\,d x-\la\int_{\Omega}b(x)|u|^{1-\delta(x)}\,d x=0.$$
Combining the   above equality  with \eqref{eq:Nota} and Theorem 2.3 in \cite{MaOgYuAv}, we get
\begin{eqnarray*}
  0 &=&\langle I'_{\lambda}(u),u\rangle =\int_\Omega p(x)|\nabla u|^{p(x)}\,{dx} 
  -\int_\Omega a(x)q(x)|u|^{q(x)} \,dx -\la\int_{\Omega}b(x)(1-\delta(x))|u|^{1-\delta(x)}\,dx
  \\ 
   &\geq&  p^-\int_\Omega |\nabla u|^{p(x)}\,d x -q^+\int_{\Omega}a(x)|u|^{q(x)}\,dx
  -(1-\delta^+)\left( \int_\Omega|\nabla u|^{p(x)} \,dx -\int_{\Omega}a(x)|u|^{q(x)}\,d x\right) .
 \\ &\geq&  (p^- -(1-\delta^+))\int_\Omega |\nabla u|^{p(x)}\,d x 
  +(1-\delta^+-q^+)\int_{\Omega}a(x)|u|^{q(x)}\,dx. 
\end{eqnarray*}
It now follows from Theorem \ref{Inj3} that
$$ (p^- -(1-\delta^+))|| u||^{p^-}
  +c_{10}(1-\delta^+-q^+)||u||^{q^+}\geq 0,$$
  hence
  \begin{equation}\label{32eq}
  ||u||\geq c_{10}\left(\frac{p^- +\delta^+ -1}{1-\delta^+-q^+}\right)^{\frac{1}{q^+-p^-}}.
\end{equation}   
In the same way, since $u\in\mathcal{N}_\la,$ we obtain
$$
\int_\Omega|\nabla u|^{p(x)} \,dx
-\int_{\Omega}a(x)|u|^{q(x)}\,d x-\la\int_{\Omega}b(x)|u|^{1-\delta(x)}\,d x=0$$
and since $u\in\mathcal{N}_\la^{0},$ we have
$$
p^+\int_\Omega|\nabla u|^{p(x)} \,dx
- q^-\int_{\Omega}a(x)|u|^{q(x)}\,d x-\la(1-\delta^+)\int_{\Omega}b(x)|u|^{1-\delta(x)}\,d x\geq0.$$
Therefore
\begin{eqnarray*}
&&p^+\int_\Omega|\nabla u|^{p(x)} \,dx
- q^-\int_{\Omega}a(x)|u|^{q(x)}\,d x\\
&&-\la(1-\delta^+)\left(\int_\Omega|\nabla u|^{p(x)} \,dx
-\la\int_{\Omega}a(x)|u|^{q(x)}\,d x\right)\geq0.\\
&&=(p^+-q^+)\int_\Omega|\nabla u|^{p(x)} \,dx+\lambda(q^++\delta^+-1)\int_{\Omega}b(x)|u|^{1-\delta(x)}\,d x\geq 0.
\end{eqnarray*}

Now since $||u||>1$, by Theorem \ref{Inj3}, one has
 \begin{eqnarray*}
(p^+-q^+)|| u||^{p^-}+c_{11}\lambda(q^++\delta^+-1)||u||^{1-\delta^+} \geq 0,
\end{eqnarray*}
and therefore
 \begin{eqnarray}\label{33eq}
|| u||\leq c_{11}\left(\lambda \frac{q^++\delta^+-1}{q^+-p^-}\right)^{\frac{1}{p^-+\delta^+-1}}.
\end{eqnarray}

Using \eqref{32eq} and \eqref{33eq},
\begin{equation*}
 c_{11}\left(\lambda \frac{q^++\delta^+-1}{q^+-p^-}\right)^{\frac{1}{p^-+\delta^+-1}}\geq c_{10}\left(\frac{p^- +\delta^+ -1}{1-\delta^+-q^+}\right)^{\frac{1}{q^+-p^-}}.
\end{equation*}
we get
\begin{equation*}
  \lambda\geq \frac{c_{10}}{c_{11}}\left(\frac{p^- +\delta^+ -1}{1-\delta^+-q^+}\right)^{\frac{p^-+\delta^+-1}{q^+-p^-}} \left( \frac{q^++\delta^+-1}{q^+-p^-}\right).
\end{equation*}

Then, if $\lambda$ is  small enough,
$$\lambda =  \frac{c_{10}}{c_{11}}\left(\frac{p^- +\delta^+ -1}{1-\delta^+-q^+}\right)^{\frac{p^-+\delta^+-1}{q^+-p^-}} \left( \frac{q^++\delta^+-1}{q^+-p^-}\right),$$ 
we obtain $||u||<1$, which is impossible. Therfore, $\mathcal{N}_{\lambda}^{0}=\{0\}$ for all $\lambda\in (0,\lambda_0).$
Hence, this completes the proof of  Lemma \ref{lem2.2}.
\end{proof}

\section{Existence of minimizers on $\mathcal{N}_{\lambda}^{+}$}\label{existence1}
In this section, we shall prove the existence of a minimum for the functional energy $E_\la$ in
$\mathcal{N}_{\lambda}^{+}$. We shall also  prove  that this minimizer is a solution to problem 
 $(P_\la).$

\begin{theorem}\label{mini1}
There exists $u_\lambda \in \mathcal{N}_{\lambda}^{+}$
satisfying 
$$E_\lambda(u_\lambda) = \inf_{u\in \mathcal{N}_{\lambda}^{+}} E_\lambda(u),$$
for all $\lambda\in (0,\lambda_0).$
\end{theorem}

\begin{proof}
Suppose that  $\lambda \in (0, \lambda_0)$. Now,  $E_\lambda$ is bounded below on $\mathcal{N}_{\lambda}$ and hence also on $\mathcal{N}_{\lambda}^{+}.$ Therefore  there exists a sequence $\{u_{n}\} \subset \mathcal{N}_{\lambda}^{+},$   satisfying
$E_\la(u_{n}) \to \inf_{u\in \mathcal{N}_{\lambda}^{+}} E_\la(u),$ as $n \to \infty$.

 Since $E_\lambda$ is coercive, $\{u_n\}$ is bounded in $W^{1,p(x)}_0(\Omega).$ Therefore we can assume that  $u_n \rightharpoonup u_0$ weakly  in $W^{1,p(x)}_0(\Omega)$ and by the compact embedding, we obtain
$$u_n \rightharpoonup u_0\;\textrm{ in}\; L^{1-\delta(x)}_{b(x)}(\Omega)$$
 and $$u_n \rightharpoonup u_0\;\textrm{ in}\; L^{q(x)}_{a(x)}(\Omega).$$ 
 
 Now, we shall show that
 $u_n \to u_0$ strongly in $W^{1,p(x)}_0(\Omega).$ 
  First, we shall prove that
$$\inf_{u\in \mathcal{N}_{\lambda}^{+}} E_\lambda(u) < 0.$$ 

Let $u_0 \in \mathcal{N}_{\lambda}^{+}.$
Then  $\phi''_{u_0}(1) >0$ which gives
\begin{equation}\label{1eq}
p^+\int_\Omega |\nabla u|^{p(x)}\,{d x} 
  - q^-\int_\Omega a(x) |u|^{q(x)} \,dx - \la(1-\delta^+)\int_{\Omega} b(x)|u|^{1-\delta(x)}\,dx>0.
\end{equation}

Moreover, by the definition of the functional energy $E_\la,$ we can write 
\begin{equation}\label{2eq}
E_\la(u)\leq \frac{1}{p^-}\int_\Omega|\nabla u|^{p(x)} \,dx -\frac{1}{q^+}\int_\Omega a(x) |u|^{q(x)}\,dx
-\frac{\la}{1-\delta^+}\int_{\Omega}b(x)|u|^{1-\delta(x)}\,dx.
\end{equation}

Now, we multiply \eqref{eq:Nota} by $(-(1-\delta^+))$ and get
\begin{eqnarray*}
-(1-\delta^+)\int_\Omega |\nabla u|^{p(x)}\,dx 
  +(1-\delta^+)\int_\Omega a(x) |u|^{q(x)} \,dx + \lambda(1-\delta^+) \int_{\Omega} b(x)|u|^{1-\delta(x)}\,dx=0.
\end{eqnarray*}

Invoking the above equality and \eqref{1eq}, one gets
\begin{equation}\label{3eq}
\int_{\Omega} a(x)|u|^{q(x)}\,dx<\frac{p^++\delta^+-1}{q^-+\delta^+-1}\int_\Omega |\nabla u|^{p(x)}\,dx .
\end{equation}

On the other hand, from \eqref{eq:Nota} and \eqref{2eq}, we obtain
\begin{equation}\label{4eq}
E_\la(u)\leq (\frac{1}{p^-}-\frac{1}{1-\delta^+})\int_\Omega|\nabla u|^{p(x)} \,dx 
-(\frac{1}{q^-}-\frac{1}{1-\delta^+})\int_{\Omega}a(x)|u|^{q(x)}\,d x.
\end{equation}

Then, by \eqref{3eq} and \eqref{4eq}, we  get
\begin{equation*}
E_\la(u)<- \frac{(p^-+\delta^+-1)(q^+-p^-)}{p^-q^+(1-\delta^+)}|| u||^{p^-}<0.
\end{equation*}

Now, let us assume that  $u_n \nrightarrow u_0$ strongly in $W^{1,p(x)}_0(\Omega).$ Then
\begin{equation*}
\int_\Omega|\nabla u_0|^{p(x)}\,dx\leq \underset{n\to\infty}{\lim\inf}\int_\Omega
|\nabla u_n|^{p(x)}\,dx.
\end{equation*}
Using the compactness of embeddings,  we obtain
\begin{equation*}
\int_\Omega a(x) u_0^{q(x)}\, dx=\underset{n\to\infty}{\lim\inf}\int_\Omega b(x)u_n^{q(x)}\,dx,
\end{equation*}
\begin{equation*}
\int_\Omega b(x) u_0^{1-\delta(x)}\, dx=\underset{n\to\infty}{\lim\inf}\int_\Omega a(x)u_n^{1-\delta(x)}\,dx.
\end{equation*}

Now, by \eqref{eq:Nota} and Theorem 2.3 in \cite{MaOgYuAv}, one has
\begin{equation*}
E_\la(u_n)\geq \left(\frac{1}{p^-}-\frac{1}{q^+}\right)\int_\Omega|\nabla u_n|^{p(x)} \,dx +\lambda\left(\frac{1}{q^+}-\frac{1}{1-\delta^+}\right)\int_\Omega b(x) |u_n|^{1-\delta(x)}\,dx.
\end{equation*}
Passing to the limit when  $n$ goes to $\infty,$ we obtain
\begin{eqnarray*}
&&
\underset{n\to\infty}{\lim} E_\la(u_n)\geq \left(\frac{1}{p^-}-\frac{1}{q^+}\right)\underset{n\to\infty}{\lim}\int_\Omega|\nabla u_n|^{p(x)} \,dx \nonumber\\
&&+\lambda\left(\frac{1}{q^+}-\frac{1}{1-\delta^+}\right)\underset{n\to\infty}{\lim}\int_\Omega b(x) |u_n|^{1-\delta(x)}\,dx.
\end{eqnarray*}

Hence, using Theorem 2.3 in \cite{MaOgYuAv}, we get
\begin{eqnarray*}
\underset{u\in \mathcal{N}^+}{\inf} E_\la(u)> \left(\frac{1}{p^-}-\frac{1}{q^+}\right)
|| u_0||^{p^-}  +\lambda c_5\left(\frac{1}{q^+}-\frac{1}{1-\delta^+}\right) (||u_0||^{1-\delta^-}+||u_0||^{1-\delta^+})>0
\end{eqnarray*}
since $p^->1-\delta^+\geq 1-\delta^- $ and $||u_0||>1$, 
which gives a contradiction. Therefore, $u_n \to u_0$ strongly in $W^{1,p(x)}_0(\Omega)$ and
$E_\lambda(u_0) = \inf_{u\in \mathcal{N}_{\lambda}^{+}} E_\lambda(u)$. This completes the proof of  
Theorem \ref{mini1}.
\end{proof}

\section{Existence of minimizers on $\mathcal{N}^{-}_{\lambda}$}\label{existence2}
In this section, we shall prove the existence of a minimum for the functional energy $E_\la$ in
$\mathcal{N}_{\lambda}^{-}$. We shall also
  prove  that this minimizer is a solution to problem 
 $(P_\la).$ 
\begin{theorem}\label{mini2}
There exists $v_\lambda \in \mathcal{N}_{\lambda}^{-}$
such that $$E_\lambda(v_\lambda) = \inf_{v\in \mathcal{N}_{\lambda}^{-}} E_\lambda(v),$$
for all $\lambda\in (0,\lambda_0).$
\end{theorem}

\begin{proof}
Suppose that  $\lambda \in (0, \lambda_0)$. Since  $E_\lambda$ is bounded below on $\mathcal{N}_{\lambda}$ hence also on $\mathcal{N}_{\lambda}^{-}.$ Therefore, there exists  a sequence
$\{v_{n}\} \subset \mathcal{N}_{\lambda}^{-},$  satisfying
$E_\la(v_{n}) \to \inf_{u\in \mathcal{N}_{\lambda}^{-}} E_\la(u),$ as $n \to \infty$. Since $E_\lambda$ is coercive, $\{v_n\}$ is bounded in $W^{1,p(x)}_0(\Omega).$ 

Therefore we can assume that $v_n \rightharpoonup v_0$ weakly in $W^{1,p(x)}_0(\Omega)$ and by the compact embedding, we get
$$v_n \rightharpoonup v_0\;\textrm{ in}\; L^{1-\delta(x)}_{b(x)}(\Omega)$$
 and $$v_n \rightharpoonup v_0\;\textrm{ in}\; L^{q(x)}_{a(x)}(\Omega).$$ 
 
 Now, we shall show 
 $v_n \to v_0$ strongly in $W^{1,p(x)}_0(\Omega).$ First, we shall prove that
$$\inf_{v\in \mathcal{N}_{\lambda}^{-}} E_\lambda(v) > 0.$$ 

Let $v_0 \in \mathcal{N}_{\lambda}^{-}$. Then we have from \eqref{eq:Nota}, 
\begin{equation}\label{11eq}
\int_\Omega |\nabla u|^{p(x)}\,dx 
  - \int_\Omega a(x) |u|^{q(x)} \,dx - \la\int_{\Omega} b(x)|u|^{1-\delta(x)}\,dx=0.
\end{equation}

Moreover, by the definition of the functional energy $E_\la,$  we can write 
\begin{equation}\label{12eq}
E_\la(v)\geq \frac{1}{p^-}\int_\Omega|\nabla v|^{p(x)} \,dx -\frac{1}{q^+}\int_\Omega a(x) |v|^{q(x)}\,d x
-\frac{\la}{1-\delta^+}\int_{\Omega}b(x)|v|^{1-\delta(x)}\,d x.
\end{equation}

Hence, from \eqref{11eq} and \eqref{12eq}, one has
\begin{eqnarray*}
&&E_\la(v)\geq \frac{1}{p^-}\int_\Omega|\nabla v|^{p(x)} \,dx -\frac{\lambda}{1-\delta^+}\int_\Omega b(x) |v|^{1-\delta(x)}\,d x\nonumber\\
&&-\frac{1}{q^+} \left(\int_\Omega |\nabla v|^{p(x)}\,dx 
  -\lambda \int_\Omega b(x) |v|^{1-\delta(x)} \,dx\right)\nonumber\\
 && \geq \left(\frac{1}{p^-}-\frac{1}{q^+}\right)\int_\Omega|\nabla v|^{p(x)} \,dx 
  +\left(\frac{1}{q^+}-\frac{1}{1-\delta^+}\right)\int_\Omega b(x) |v|^{1-\delta(x)}\,d x
  \nonumber\\
&&  \geq \left(\frac{1}{p^-}-\frac{1}{q^+}\right)|| v||^{p^-}  
  +\lambda c_8\left(\frac{1}{q^+}-\frac{1}{1-\delta^+}\right)||v||^{1-\delta^+}\nonumber\\
 && \geq \left[\left(\frac{1}{p^-}-\frac{1}{q^+}\right)  
  + \lambda c_8\left(\frac{1}{q^+}-\frac{1}{1-\delta^+}\right)\right]||v||^{p-}
\end{eqnarray*}
since $p^->1-\delta^+.$ 

Hence, if we choose $$\lambda<\frac{(1-\delta^+)(p^--q^+)}{c_8p^+(1-\delta^+-q^+)},$$
 we obtain $E_\la(v)>0.$ Moroever, since $\mathcal{N}_\lambda^+\cap\mathcal{N}_\lambda^-=\emptyset$
and
 $\inf_{v\in \mathcal{N}_{\lambda}^{+}} E_\lambda(v) < 0$, 
 we see that $v\in \mathcal{N}_{\lambda}^{-}.$

In the same way,  if $v_0\in \mathcal{N}_{\lambda}^{-}$, hence  there exists $t_0$ satisfying $t_0v_0 
\in \mathcal{N}_{\lambda}^{-}$ and so $E_\lambda(t_0v_0)\leq E_\lambda(v_0).$ 
Moreover, since 
\begin{eqnarray*}
 I'_{\lambda}(v)=\int_\Omega p(x)|\nabla v|^{p(x)}\,dx 
  -\int_\Omega a(x)q(x)|v|^{q(x)} \,dx -\la\int_{\Omega}b(x)(1-\delta(x))|v|^{1-\delta(x)}\,dx,
\end{eqnarray*}
we get
\begin{eqnarray*}
 &&I'_{\lambda}(t_0v_0)=\int_\Omega p(x)|\nabla t_0v_0|^{p(x)}\,dx 
  -\int_\Omega a(x)q(x)|t_0v_0|^{q(x)} \,dx -\la\int_{\Omega}b(x)(1-\delta(x))|t_0v_0|^{1-\delta(x)}\,dx
  \\ 
   &&\leq t_0^{p^+}p^+\int_\Omega |\nabla v_0|^{p(x)}\,dx 
  - t_0^{q^-}q^-\int_\Omega a(x)|v_0|^{q(x)} \,dx -\la(1-\delta^+)t_0^{1-\delta^+}\int_{\Omega}b(x)|v_0|^{1-\delta(x)}\,dx,
\end{eqnarray*}
since $1-\delta^+ < p^+ < q^-$. By the conditions on $a$ and $b$ it follows that
$I'_{\lambda}(t_0v_0)<0,$ so by definition of $\mathcal{N}_{\lambda}^{-}$ $t_0v_0\in \mathcal{N}_{\lambda}^{-}.$

Now, let us assume that  $v_n \nrightarrow v_0$ strongly in $W^{1,p(x)}_0(\Omega).$ Using the fact that
\begin{equation*}
\int_\Omega|\nabla v_0|^{p(x)}\,dx\leq \underset{n\to\infty}{\lim\inf}\int_\Omega,
|\nabla v_n|^{p(x)}\,dx
\end{equation*}
one gets
\begin{eqnarray*}
&&E_\la(tv_0)\leq \int_\Omega\frac{t^{p(x)}|\nabla v_0|^{p(x)}}{p(x)}\,{\rm d }x
 -\int_\Omega \frac{t^{q(x)}}{q(x)}  |v_0|^{q(x)} \,dx-\la\int_{\Omega}\frac{t^{1-\delta(x)}|v_0|^{1-\delta(x)}}{1-\delta(x)}\,dx,\\
&&\leq \underset{n\to\infty}{\lim}\left[\frac{t^{p^+}}{p^-}\int_\Omega|\nabla v_n|^{p(x)}\,{\rm d }x
 -\frac{t^{q^-}}{q^+}\int_\Omega  |v_n|^{q(x)} \,dx-\la\frac{t^{1-\delta^+}}{1-\delta^+}\int_{\Omega}|v_n|^{1-\delta(x)}\,dx\right],\\
&&\leq \underset{n\to\infty}{\lim}E_\la(tv_n)\leq \underset{n\to\infty}{\lim}E_\la(v_n) 
=  \underset{v 
\in \mathcal{N}_{\lambda}^{-}}{\inf}E_\la(v),
\end{eqnarray*}
which contradicts with the fact that  $tv_0 \in \mathcal{N}_{\lambda}^{-}.$ 
Hence, $v_n \to v_0$ strongly in $W^{1,p(x)}_0(\Omega)$ and
$E_\lambda(v_0) = \inf_{v\in \mathcal{N}_{\lambda}^{-}} E_\lambda(v)$. This completes the proof of  
Theorem \ref{mini2}.
\end{proof}

\section{Proof of  Theorem \ref{th1}}\label{proof}

\begin{proof}
 By Theorem \ref{mini1}
 and Theorem \ref{mini2}, 
 for all $\lambda\in (0,\lambda_0)$, there exist $u_0 \in \mathcal{N}_{\lambda}^{+}, $ $v_0 \in \mathcal{N}_{\lambda}^{-} $
such that $$E_\lambda(u_0) = \inf_{u\in \mathcal{N}_{\lambda}^{+}} E_\lambda(u)$$
and  $$E_\lambda(v_0) = \inf_{v\in \mathcal{N}_{\lambda}^{-}} E_\lambda(v).$$

On the other hand, since  $E_\lambda(u_0) = E_\lambda(|u_0|)$ and $|u_0|\in \mathcal{N}_{\lambda}^{+}$
 and in the same way, $E_\lambda(v_0) = E_\lambda(|v_0|)$ and  $|v_0|\in \mathcal{N}_{\lambda}^{-}$,
we suppose that
 $u_0, v_0 \geq 0.$ Using Lemma \ref{lem1}, $u_0,\,v_0$
 are critical points of $E_\lambda$ on $W^{1,p(x)}_0(\Omega)$
and thus weak solutions of $({\rm P}_\lambda).$ 

Finally, by the Harnack inequality 
and by 
 {\sc Zhang-Liu}
\cite{ZhLi}, we obtain that $u_0, v_0$ are nonnegative solutions of $({\rm P}_\lambda).$ 

It remains to prove that
the solutions we obtained for Theorem \ref{mini1}
and Theorem \ref{mini2} are distinct. Indeed, since $\mathcal{N}_{\lambda}^{-}\cap \mathcal{N}_{\lambda}^{+}=\emptyset$, it follows that $u_0$ and $v_0$ are different. This completes the proof of  Theorem \ref{th1}.
\end{proof}

\section*{Acknowledgements}
The first author was supported by the Slovenian Research Agency program P1-0292 and grants N1-0114 and N1-0083.


\begin{thebibliography}{99}

\bibitem{AcMi} {\sc E. Acerbi and G. Mingione}, {\it Regularity results for a class of functionals with nonstandard growth}, 
Arch. Rational Mech. Anal. {\bf 156} (2001) 121--140. 

\bibitem{CoPa} {\sc M.M. Coclite and G. Palmieri,}
{\it On a singular nonlinear Dirichlet problem}, 
Comm. Partial Differential Equations {\bf 14} (1989), 1315--1327.

\bibitem{CrRaTa} {\sc M.G. Crandall, P.H. Rabinowitz and L. Tartar}, 
{\it On a Dirichlet problem with a singular nonlinearity}, 
Comm. Partial Differential Equations {\bf 2} (1977), 193-222.

\bibitem{Di} {\sc L. Diening}, {\it Theorical and numerical results for electrorheological fluids}, Ph.D. Thesis, University of
Freiburg, Freiburg, 2002.

\bibitem{DrPo} {\sc P. Drabek  and S.I. Pohozaev,}{\it Positive solutions for the p-Laplacian: application of the fibering method}, Proc. Royal Soc. Edinburgh Sect A, {\bf 127} (1997), 703--726.

\bibitem{F2} {\sc X. L. Fan,}
{\it Solutions for $p(x)$-Laplacian Dirichlet problems
with singular coefficients,} J. Math. Anal. Appl. {\bf 312} (2005), 464--477. 

\bibitem{Fan1} {\sc X.L. Fan, J.S. Shen and D. Zhao}
\textit{\ Sobolev embedding Theorems for spaces $W^{k,p(x)} (\Omega)$}, 
J. Math. Anal. Appl. {\bf 262} (2001), 749--760.

\bibitem{f8} {\sc X.L. Fan and D. Zhao,}
{\it On spaces $L^{p(x)}( \Omega ) $ and $W^{m,p(x)}( \Omega ) $,}
J. Math. Anal. Appl. {bf 263} (2001), 424--446.

\bibitem{GhRa} {\sc M. Ghergu and V. R\u{a}dulescu,} 
{Singular Elliptic Problems: Bifurcation and Asymptotic Analysis,} Oxford Lecture Series in Mathematics and its
Applications {\bf 37}.  Oxford University Press, Oxford,
2008.

\bibitem{GiSa} {\sc J. Giacomoni and K. Saoudi,} {\it Multiplicity of positive solutions for a singular and critical problem,} 
Nonlinear Anal. {\bf 71} (2009), no. 9, 4060--4077.

\bibitem{GiScTa} {\sc J. Giacomoni, I. Schindler and P. Tak\'{a}\v{c}}, 
{\it Sobolev versus H\"{o}lder local minimizers and
global multiplicity for a singular and quasilinear equation},
Annali della Scuola Normale Superiore di Pisa, Classe di scienze S\'erie V {\bf 6}  (2007), 117--158.

\bibitem{Kovacik} {\sc O. Kov\u{a}\v{c}ik and J. R\u{a}kosnik,}
{\it On spaces $L^{p(x)}$ and $W^{k,p(x)}$}, 
Czechoslovak Math. J. {\bf 41} (116) (1991), 592--618.

\bibitem{MaOgYuAv}{\sc R.A. Mashiyev, S. Ogras, Z. Yucedag and M. Avci,}
{\it  The Nehari manifold approach for Dirichlet problem involving the $p(x)$-Laplacian equation,} J. Korean Math. Soc. {\bf 47}
(2010), No. 1. pp. 1--16.

\bibitem{PRR}{\sc N.S. Papageorgiou, V.D. R\u{a}dulescu, D.D. Repov\v{s},}
Nonlinear Analysis - Theory and Methods, Springer Monographs in Mathematics,
Springer, Cham, 2019. 

\bibitem{RR}
{\sc V.D. R\u{a}dulescu and D.D. Repov\v{s},}
Partial Differential Equations with Variable Exponents: Variational Methods and Qualitative Analysis,
Chapman and Hall\ /CRC, Taylor \& Francis Group, Boca Raton, FL, 2015.

\bibitem{SaiGh}{\sc S. Saiedinezhad and M.B. Ghaemi}, {\it The fibering map approach to a quasilinear degenerate $p(x)$-Laplacian equation,} Bull. Iranian Math. Soc., {\bf 41} (2015), 1477-1492.

\bibitem{Sa5}{\sc K. Saoudi}, {\it Existence and non-existence of  solution  for a singular nonlinear Dirichlet problem involving the $p(x)$-Laplace operator,} J. Adv. Math. Stud. {\bf 9} (2016), No. 2, 292-303.

\bibitem{Sa3}{\sc K. Saoudi,} {\it Existence and non-existence  for a singular problem with variables potentials,} Electronic Journal of Differential equations {\bf 2017} (2017), No. 291, 1-9.

\bibitem{Sa4}{\sc Saoudi,} {\it A singular elliptic system involving the p(x)-Laplacian and generalized Lebesgue–Sobolev spaces,} International Journal of Mathematics {\bf 30}(116):1950064 (2019).

\bibitem{SaGh} {\sc  K. Saoudi and  A. Ghanmi,} {\it A multiplicity results for a singular   problem  involving the fractional $p$-Laplacian operator,} Complex Var. Elliptic Equ. {\bf 61}:9 (2016), 1199--1216.

\bibitem{SaGh1} {\sc  K. Saoudi and  A. Ghanmi,} {\it A multiplicity results for a singular  equation involving the $p(x)$-Laplace operator},  Complex Var. Elliptic Equ. {\bf 62} (2017) No. 5,  695-725.

\bibitem{SaKr}{\sc K. Saoudi and M. Kratou,} {\it Existence of multiple solutions
for a singular and quasilinear equation,}  Complex Var. Elliptic Equ. {\bf 60} (2015), no.7, 893--925.

\bibitem{SaKrAl}{\sc K. Saoudi, M. Kratou and S. Al Sadhan,} {\it Multiplicity results for the $p(x)$-Laplacian equation
with singular nonlinearities and nonlinear Neumann boundary condition,}  Intern. J. Differential Equ. {\bf 2016} (2016), Article ID 3149482, 14 pp.

\bibitem{z1} {\sc Q.H. Zhang,}
{\it A strong maximum principle for differential
equations with nonstandard $p(x)$-growth conditions,} J. Math. Anal.
Appl. \textbf{312} (2005), No.1, 24--32.

\bibitem{Zh1} {\sc Q. Zhang,}
{\it Existence and asymptotic behavior of positive solutions to $p(x)$-Laplacian equations with singular nonlinearities,} J. Inequal. Appl.  2007, Art. ID 19349, 9 pp.

\bibitem{ZhLi} {\sc X. Zhang and  X. Liu,} {\it  The local boundedness and Harnack inequality of $p(x)$-
Laplace equation,} J. Math. Anal. Appl. {\bf 332} (2007), 209--218. 

\end{thebibliography}
\end{document}